\documentclass[a4paper,12pt]{article}
\usepackage{amsmath,amsthm,amssymb}
\usepackage{xcolor}

\usepackage{graphicx}
\usepackage{epstopdf}
\usepackage{indentfirst}
\usepackage{geometry}
\usepackage{caption}
\usepackage{enumerate}
\usepackage[numbers,sort&compress]{natbib}
\usepackage{float}
\usepackage[hidelinks]{hyperref}
\usepackage[nameinlink]{cleveref}

\theoremstyle{definition}
\newtheorem{The}{Theorem}
\newtheorem{Lem}[The]{Lemma}
\newtheorem{Cor}[The]{Corollary}
\newtheorem{Pro}[The]{Proposition}
\newtheorem{Cla}{Claim}
\newtheorem{Subcla}{Subclaim}[Cla]
\newtheorem{Prob}[The]{Problem}
\newtheorem{Cas}{Case}

\newtheorem{Not}[The]{Note}

\crefname{The}{Theorem}{Theorems}
\crefname{Lem}{Lemma}{Lemmas}
\crefname{Cor}{Corollary}{Corollaries}
\crefname{Pro}{Proposition}{Propositions}
\crefname{Cla}{Claim}{Claims}
\crefname{Subcla}{Subclaim}{Subclaims}
\crefname{Prob}{Problem}{Problems}
\crefname{Cas}{Case}{Cases}
\crefname{Subcas}{Subcase}{Subcases}
\crefname{Not}{Note}{Notes}
\crefname{figure}{Fig.}{Figs.}

\begin{document}
\title{Near-bipartite bricks in which every $b$-invariant edge is a forcing edge
}
\author{
   Yaxian Zhang$^{1}$ and Fuliang Lu$^{2,3}$\thanks{Corresponding author.}
}
\date{
		{\small
         $^{1}$School of Mathematical Sciences, Sichuan Normal University, 
         Chengdu, Sichuan 610068, P.R. China\\
			$^{2}$School of Mathematics and Statistics, Minnan Normal University,
			Zhangzhou, Fujian 363000, P.R. China\\
			$^{3}$Fujian Key Laboratory of Granular Computing and Applications,
			Minnan Normal University, Zhangzhou, Fujian 363000, P.R. China\\[0.5ex]
			E-mails: yaxianzhang@sicnu.edu.cn, flianglu@163.com
		}
	}

\maketitle
\begin{abstract}
   A connected graph is matching covered if it has at least one edge and every edge lies in some perfect matching.
   Lov\'asz proved that every matching covered graph $G$ can be uniquely decomposed into a list of bricks
   and braces up to multiple edges. Denote by $b(G)$ the number of bricks in such a decomposition.
   An edge $e$ of $G$ is removable if $G-e$ is also matching covered; it is $b$-invariant if $e$ is removable and $b(G-e)=b(G)$.
   Furthermore, an edge $e$ of $G$ is a forcing edge if it lies in precisely one perfect matching of $G$.

   Lucchesi and Murty proposed the problem of characterizing bricks, distinct from
$K_4$, $\overline{C_6}$, and the Petersen graph, in which every $b$-invariant
edge is a forcing edge. In this paper, we solve this problem for near-bipartite bricks
by providing a complete characterization.

   \medskip
\end{abstract}
\section{Introduction}
All graphs in this paper are finite and contain no loops (multiple edges are allowed).
The \emph{underlying graph} of $G$ is the simple graph obtained from $G$
by deleting all but one of the edges joining each pair of adjacent vertices. 
A subset of the edge set is a \emph{perfect matching} of a graph if
every vertex is incident with exactly one edge of the subset.
A connected graph is \emph{matching covered} if it contains at least one edge,
and each edge belongs to some perfect matching.
For the terminology that is specific to matching covered graphs,
we follow Lov\'asz and Plummer \cite{lovasz2009matching}.

Let $G$ be a connected graph with vertex set $V(G)$ and edge set $E(G)$.
For any nonempty subset $X\subset V(G)$, we call $\partial(X)$ an \emph{edge cut} of $G$, which is the set of edges with one end vertex in $X$ and the other in $\overline{X}=V(G)-X$.
In particular, if $|X|=1$ or $|\overline{X}|=1$, then the edge cut $\partial(X)$ is called \emph{trivial}; otherwise, \emph{nontrivial}.
An edge cut of $G$ is a \emph{tight cut}, if it intersects each perfect matching of $G$ exactly in one edge.
A matching covered graph without nontrivial tight cuts is a \emph{brick}, if it is nonbipartite; a \emph{brace}, otherwise.

A graph $G$ is \emph{bicritical} if $G-x-y$ has a perfect matching for any two distinct vertices $x,y\in V(G)$. 
Edmonds et al. \cite{Edmonds1982} showed that
a graph $G$ is a brick if and only if it is 3-connected and bicritical 
(also see Szigeti \cite{Szigeti2002} and de Carvalho et al. \cite{Carvalho2018}).
Lov\'asz \cite{Lovasz1987} proved that every matching covered graph can be decomposed into a unique list of bricks and braces up to multiple edges by a procedure
called the tight cut decomposition. In particular, any two applications of the tight cut
decomposition of a matching covered graph $G$ yield the same number of bricks, which is
called the brick number of $G$ and denoted by $b(G)$.

An edge $e$ of a matching
covered graph $G$ is removable if $G-e$ is also matching covered; it is $b$-invariant if $e$ is removable and $b(G-e)=b(G)$.
Confirming a conjecture of Lov\'asz, de Carvalho et al. \cite{Carvalho2002II} showed that every brick distinct from $K_4$, $\overline{C_6}$ and the Petersen graph has a $b$-invariant edge.

Kothari et al. \cite{KCLL20} showed that every essentially 4-edge-connected cubic non-near-bipartite brick $G$, distinct from the Petersen graph, has at least $|V(G)|$ $b$-invariant edges.
Moreover, they conjectured that every essentially 4-edge-connected cubic near-bipartite brick $G$, distinct from $K_4$, has at least $|V(G)|/2$ $b$-invariant edges;
Lu et al. \cite{LXY19} confirmed this conjecture.
A brick $G$ is {\em solid} if $G-(V(C_1)\cup V(C_2))$ has no perfect matching for any two vertex-disjoint odd cycles $C_1$ and $C_2$.
De Carvalho et al. \cite{CLM12} proved that every solid brick, distinct from $K_4$, has at least $\frac{|V(G)|}{2}$ $b$-invariant edges.

Lucchesi and Murty \cite{Lucchesi2022} used $b$-invariant edges to characterize extremal matching covered graphs and introduced the notion of a solitary edge.
An edge of a graph is \emph{solitary} if it lies in precisely one perfect matching.
In fact, solitary edges had already been studied in benzenoid hydrocarbons in theoretical chemistry under the name ``\emph{forcing edge}''.
This concept was first introduced by Klein and Randi\'c \cite{Klein1987} under the
name \emph{innate degree of freedom} and was later called \emph{forcing} by Harary et al.
\cite{Harary1991}. For hexagonal systems and polyomino graphs, Zhang et al.
\cite{zhang1988,zhang1995} and Sun et al. \cite{sun2024}, respectively,
characterized the graphs whose resonance graphs have a vertex of degree one and,
as a consequence, determined all such graphs with a forcing edge. For
3-connected cubic graphs, Wu et al. \cite{Ye2016} showed that every graph with a
forcing edge can be generated from $K_4$ by repeatedly replacing a vertex with a
triangle. Recently, Goedgebeur et al. \cite{Goedgebeur2026} studied forcing edges in bridgeless cubic graphs.

A matching covered graph $G$ is \emph{extremal} if the number of its perfect matchings
equals the dimension of the lattice spanned by their incidence vectors.
For a brick $G$, this lattice has dimension
$|E(G)|-|V(G)|+1$. Consequently, a brick $G$ is extremal if and only if
$|\mathcal{M}(G)|=|E(G)|-|V(G)|+1$,
where $\mathcal{M}(G)$ denotes the set of all perfect matchings of $G$. Using solitary $b$-invariant edges, de Carvalho et al. \cite{Carvalho2004} proved that a matching covered graph $G$ is extremal if and
only if every $b$-invariant edge $e$ of $G$ is solitary and $G-e$ is extremal. Motivated
by this connection between $b$-invariant and solitary edges, Lucchesi and Murty posed
the following more general problem.

\begin{Prob}
   \rm \cite{Lucchesi2022}
   \label{prob1}
Characterize bricks, distinct from $K_4$, $\overline{C_6}$ and the Petersen graph, in which every
$b$-invariant edge is solitary.
\end{Prob}

\cref{prob1} has been settled for several classes of bricks. In particular, in extremal bricks every
$b$-invariant edge is a forcing edge \cite{Carvalho2004}.
For simple solid bricks, Zhang et al. \cite{ZhangW25} proved that every such brick satisfying this property is an odd wheel.
For claw-free bricks, Zhang and Wang \cite{ZhangWClaw25} proved that exactly four graphs have the property described in \cref{prob1}. 
For cubic bricks, Zhang et al. \cite{Zhang24} obtained the following characterization.
\begin{The}
   \label{th1}
   Let $G$ be a cubic brick other than $K_4$, $\overline{C_6}$ and the Petersen graph. All $b$-invariant edges in $G$
   are forcing edges if and only if $G$ is one of the graphs in $\{G_i:1\leq i\leq 7\}$ (see \cref{f4}).
\end{The}
\begin{figure}[H]
   \centering
   \includegraphics[scale=0.3]{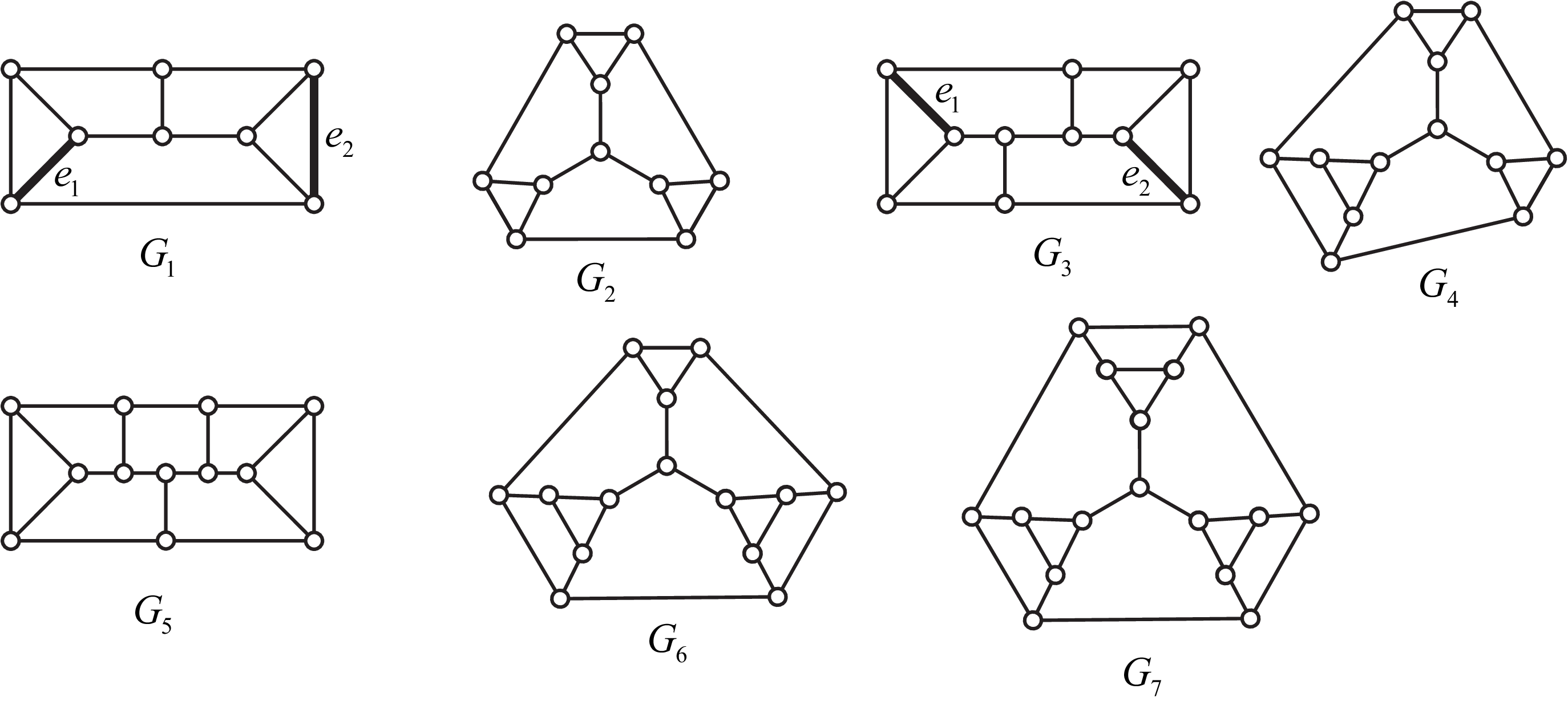}
   \caption{\label{f4} Illustration of \cref{th1}.}
\end{figure}
A pair $\{e_1,e_2\}$ of edges in a matching covered graph $G$ is a \emph{removable doubleton}
if $G-e_1-e_2$ is matching covered, but neither $G-e_1$ nor $G-e_2$ is.
A nonbipartite matching covered graph $G$ is \emph{near-bipartite} if it has a removable doubleton
whose deletion yields a bipartite matching covered graph.
Near-bipartite graphs have also been studied from several other viewpoints.
Fischer and Little \cite{FischerLittle2001} characterized Pfaffian near-bipartite
graphs in terms of forbidden subgraphs. Miranda and Lucchesi
\cite{MirandaLucchesi2010} extended this characterization to weakly
near-bipartite graphs and gave a polynomial-time recognition algorithm.
Kothari \cite{Kothari2019} obtained generation and structure theorems for
near-bipartite bricks.
\begin{Not}
   \label{note:cubic-near-bipartite}
   Among all bricks in \cref{f4}, only $G_1$ and $G_3$ are near-bipartite. Moreover, $K_4$ and $\overline{C_6}$ are near-bipartite, whereas the Petersen graph is not.
\end{Not}

Let $\mathcal F$ denote the union of the eight graph families shown in \cref{nf1}.
More recently, Zhang et al.
\cite{ZhangRemovable2025} studied removable edges in near-bipartite bricks.
Building on this work, Zhang et al. \cite{ZLZ2026} proved that every
near-bipartite brick with at least six vertices has at least
$(|V(G)|-6)/2$ $b$-invariant edges. Using this result as a principal tool, we
settle \cref{prob1} for near-bipartite bricks.

\begin{The}
   \label{th2}
   Let $G$ be a near-bipartite brick. Then every $b$-invariant edge of $G$ is a forcing edge if and only if $G\in\mathcal F$.
\end{The}
\begin{figure}[H]
   \centering
   \includegraphics[width=0.90\linewidth,keepaspectratio]{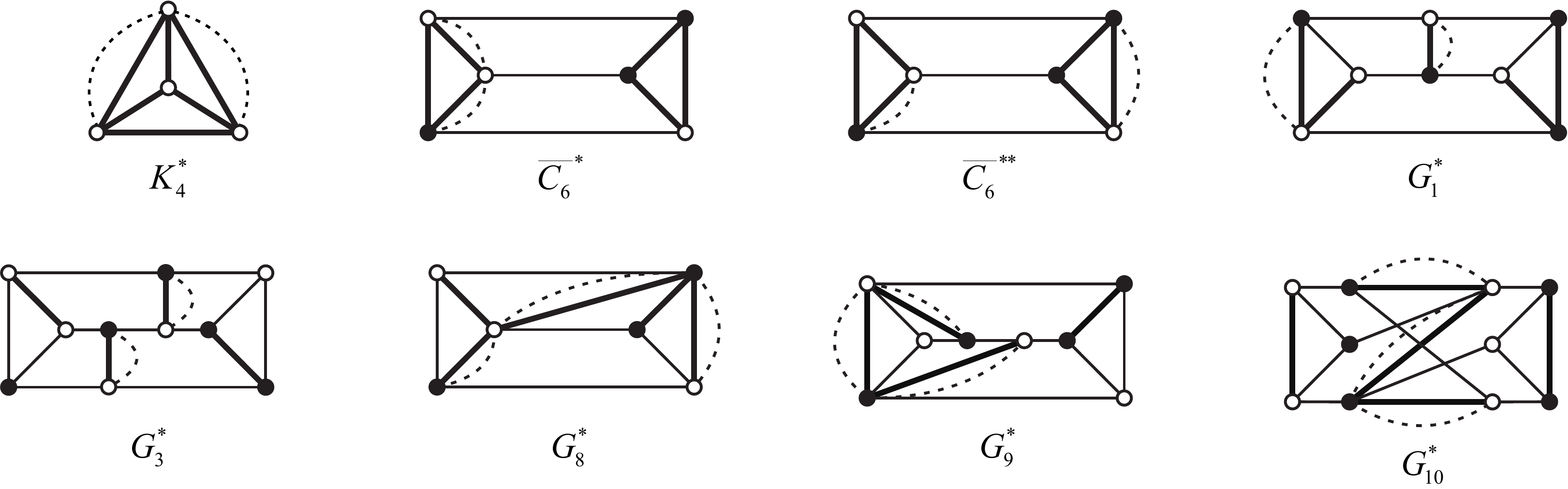}
   \caption{\label{nf1}The eight graph families in $\mathcal F$. Dashed lines indicate the edge classes that may have arbitrary positive multiplicity.}
\end{figure}
\begin{Pro}
   \label{not:extremal}
   Every graph in $\mathcal F$ is an extremal brick.
\end{Pro}
\begin{proof}
A direct check shows that the seven underlying simple graphs are bricks. Their numbers of vertices, edges and perfect matchings are as follows.
\begin{center}
\begin{tabular}{c|c|c|c}
$B$ & $|V(B)|$ & $|E(B)|$ & $|\mathcal M(B)|$ \\ \hline
$K_4$ & $4$ & $6$ & $3$ \\
$\overline{C_6}$ & $6$ & $9$ & $4$ \\
$G_1$ & $8$ & $12$ & $5$ \\
$G_3$ & $10$ & $15$ & $6$ \\
$G_8$ & $6$ & $10$ & $5$ \\
$G_9$ & $8$ & $13$ & $6$ \\
$G_{10}$ & $10$ & $16$ & $7$
\end{tabular}
\end{center}
Thus each satisfies $|\mathcal M(B)|=|E(B)|-|V(B)|+1$ and is extremal.

Let $G\in\mathcal F$, let $B$ be its underlying simple graph, and let $D\subseteq E(B)$ consist of the edge classes with multiplicity greater than one. Inspection of the perfect matchings of the graphs in \cref{nf1} shows that each edge $e\in D$ belongs to a unique perfect matching of $B$, which contains no other edge of $D$. Hence each parallel copy of $e$ belongs to exactly one perfect matching of $G$ and is therefore a forcing edge of $G$.

If $m_e$ is the multiplicity of $e$, then
\[
|\mathcal M(G)|=|\mathcal M(B)|+\sum_{e\in D}(m_e-1)
=|E(G)|-|V(G)|+1.
\]
Adding parallel copies preserves the brick property. Hence $G$ is an extremal brick.
\end{proof}

We obtain the following consequence.
\begin{Cor}
   A near-bipartite brick is extremal if and only if all its $b$-invariant edges are forcing.
\end{Cor}
\begin{proof}
   Necessity follows from \cite{Carvalho2004}, and sufficiency follows from \cref{th2,not:extremal}.
\end{proof}

\section{Preliminaries}

For the structural arguments below, let $G$ be a simple near-bipartite brick
with a fixed removable doubleton $R=\{e_1,e_2\}$, and set $H=G-R$.
Then $H$ is a bipartite matching covered graph. We fix a bipartition
$(U,V)$ of $H$ and write $e_1=u_1u_2$ and $e_2=v_1v_2$ so that
$\{u_1,u_2\}\subseteq U$ and $\{v_1,v_2\}\subseteq V$.
We use this notation throughout without further mention.

\begin{Lem}
   \cite{ZLZ2026}
   \label{lem:invariant}
   Let $G$ be a near-bipartite brick other than $K_4$.
   Then the following statements hold.
   \begin{enumerate}[(i)]
      \item
      Every vertex in $V(G)\setminus V(R)$, except at most two vertices,
      is incident with at most two non-$b$-invariant edges of $G$.
      Further, if $x\in V(G)\setminus V(R)$ is incident with at least
      three non-$b$-invariant edges of $G$, then $x$ belongs to a triangle whose vertices all have degree three.
      \item
      Every vertex in $V(R)$ is incident with at most three non-$b$-invariant edges of $G$.
   \end{enumerate}
\end{Lem}
For a vertex $x$ of a graph $G$, denote the \emph{degree} of $x$, that is, the number of edges incident with $x$, by $d_G(x)$ (or simply $d(x)$ when no confusion can arise), and denote the minimum degree of $G$ by $\delta(G)$.
The vertex $x$ is called a \emph{$k$-degree vertex} if $d_G(x)=k$.
For $\emptyset\neq S\subseteq V(G)$, let $N_G(S)$ denote the set of vertices in $V(G)\setminus S$ that have a neighbor in $S$.
When $S=\{x\}$, we write $N_G(x)$ instead of $N_G(\{x\})$.
\begin{The}
   {\rm\cite{lovasz2009matching}}
   \label{B.U.PM}
   If $G$ is a bipartite graph with a unique perfect matching,
   then $G$ has two vertices of degree one, one in each partite set.
\end{The}

\begin{Lem}
   \label{FMNR}
   Let $e$ be a forcing edge of a matching covered graph $G$, and let $M_e$ be the unique perfect matching of $G$ containing $e$.
   Then each edge in $M_e\setminus \{e\}$ is a nonremovable edge of $G$.
\end{Lem}
 
Let $U_{=3}^{e_1}$ and $V_{=3}^{e_2}$ denote the sets of all 3-degree vertices in $V(e_1)$ and $V(e_2)$, respectively.

\begin{Lem}
   \label{cl-e12}
   Let $G$ be a near-bipartite brick.
   If $G$ contains a forcing edge $e=uv$, distinct from
   $e_1$ and $e_2$, with $u\in U$ and $v\in V$, then $U_{=3}^{e_1}\neq \emptyset$ and $V_{=3}^{e_2}\neq \emptyset$.
   Moreover, $u\in N_{H-e}(V_{=3}^{e_2})$ and $v\in N_{H-e}(U_{=3}^{e_1})$.
\end{Lem}
\begin{proof}
   Since $H$ is matching covered and $e\in E(H)$,
   $e$ belongs to a perfect matching of $H$, say $M_e$. 
   Since $e$ is a forcing edge of $G$, it is also a forcing edge of $H$.
   Thus, the bipartite graph $H-u-v$ has a unique perfect matching $M_e\setminus \{e\}$.
   By \cref{B.U.PM}, $H-u-v$ contains a 1-degree vertex $u^*\in U$ and
   a 1-degree vertex $v^*\in V$, and thus $d_H(u^*)\leq 2$ and $d_H(v^*)\leq 2$.
   Since $G$ is 3-connected, $d_G(u^*)\geq 3$ and $d_G(v^*)\geq 3$.
   Since $d_G(u^*)\leq d_H(u^*)+1\leq 3$ and $d_G(v^*)\leq d_H(v^*)+1\leq 3$,
   equality holds throughout, i.e. $d_G(u^*)=d_G(v^*)=3$ and $d_H(u^*)=d_H(v^*)=2$.
   This implies that $\{u^*v,v^*u\}\subset E(H)$, and $u^*\in U_{=3}^{e_1}$ and $v^*\in V_{=3}^{e_2}$.
   So $u\in N_{H-e}(V_{=3}^{e_2})$ and $v\in N_{H-e}(U_{=3}^{e_1})$.
\end{proof}

The following results follow from \cref{cl-e12}.

\begin{Cor}
   \label{ea-v}
   Let $G$ be a near-bipartite brick with a vertex $x$ in $U\setminus V(R)$
   (resp. $V\setminus V(R)$).
   If $x$ is incident with a forcing edge of $G$, then it is adjacent to a vertex in $V_{=3}^{e_2}$ (resp. $U_{=3}^{e_1}$).
\end{Cor}
\begin{Cor}
   \label{N3}
   Let $G$ be a near-bipartite brick with a forcing edge $uv\in E(H)$,
   where $u\in U$ and $v\in V$.
   If $u\in \{u_1,u_2\}$, then $vu_1u_2v$ is a triangle of $G$ and $v\in N_G(U_{=3}^{e_1}\setminus \{u\})$;
   if $v\in \{v_1,v_2\}$, then $uv_1v_2u$ is a triangle of $G$ and $u\in N_G(V_{=3}^{e_2}\setminus \{v\})$.
\end{Cor}
\begin{proof}
   If $u\in \{u_1,u_2\}$, then relabel the vertices so that $u=u_1$, so $uv=u_1v\in E(G)$.
   By \cref{cl-e12}, $d_{G-u-v}(u_2)=1$ and thus $d_G(u_2)=3$ and $u_2v\in E(G)$.
   So $vu_1u_2v$ is a triangle of $G$.
   Similarly, if $v\in \{v_1,v_2\}$, then $uv_1v_2u$ is a triangle of $G$, and
   $u\in N_G(V_{=3}^{e_2}\setminus \{v\})$.
\end{proof}
\begin{Lem}
   \label{VGUV}
   Let $G$ be a near-bipartite brick in which every $b$-invariant edge is a forcing edge.
   Then $U=V(e_1)\cup N_H(V_{=3}^{e_2})$ and $V=V(e_2)\cup N_H(U_{=3}^{e_1})$. Moreover $|V(G)|\leq 12$.
\end{Lem}
\begin{proof}
   Suppose $U\setminus (V(e_1)\cup N_H(V_{=3}^{e_2}))$ contains at least one vertex, say $u$.
   If $u$ is incident with at least one $b$-invariant edge of $G$, then such an edge must be a forcing edge of $G$.
   By \cref{cl-e12}, $u\in N_H(V_{=3}^{e_2})$, a contradiction.
   So $u$ is not incident with any $b$-invariant edge of $G$.
   By \cref{lem:invariant}, $u$ belongs to a triangle, each vertex of which has degree three.
   This implies that $u\in N_H(V_{=3}^{e_2})$, a contradiction. So $U=V(e_1)\cup N_H(V_{=3}^{e_2})$.
   Similarly, $V=V(e_2)\cup N_H(U_{=3}^{e_1})$.
   Since $|U|=|V(e_1)\cup N_H(V_{=3}^{e_2})|\leq 2+4=6$, it follows that $|V(G)|=2|U|\leq 12$.
\end{proof}

\section{Proof of \texorpdfstring{\cref{th2}}{Theorem~\ref*{th2}}}
Let $M$ be a perfect matching of a graph $G$. An even cycle $C$ in $G$ is
called $M$-alternating if its edges alternate between $M$ and
$E(G)\setminus M$; equivalently, $M\cap E(C)$ is a perfect matching of $C$.
A cycle of length $k$ is called a \emph{$k$-cycle}. 
For two nonempty proper subsets $X_1$ and $X_2$ of $V(G)$, let $E_G(X_1,X_2)=\{xy\in E(G):x\in X_1\text{ and }y\in X_2\}$. 
Recall that $G_i$ is the underlying graph of $G_i^*$ for $i\in\{1,3,8,9,10\}$. 
In the following, we first prove the sufficiency of \cref{th2}.

\vspace{7pt}
\noindent
\textbf{Sufficiency.}
By \cref{not:extremal}, every graph in $\mathcal F$ is extremal. Since every $b$-invariant edge of an extremal brick is forcing \cite{Carvalho2004}, sufficiency follows. In particular, $K_4$ and $\overline{C_6}$ have no $b$-invariant edges and satisfy the condition vacuously. The Petersen graph also has no $b$-invariant edges, but is not near-bipartite and hence is outside the scope of this theorem.

   Recall that $e_1=u_1u_2$, $e_2=v_1v_2$, $\{u_1,u_2\}\subseteq U$, $\{v_1,v_2\}\subseteq V$, $U_{=3}^{e_1}=\{u_i|d_G(u_i)=3, i=1,2\}$ and $V_{=3}^{e_2}=\{v_i|d_G(v_i)=3, i=1,2\}$.
\\[7pt]\noindent
\textbf{Necessity. } Recall that $R$ is a removable doubleton. 
Let $G_0$ be the underlying simple graph of $G$. Suppressing parallel edges preserves the brick property and the removable doubleton $R$, so $G_0$ is a near-bipartite brick. Every $b$-invariant edge $e$ of $G_0$ has a $b$-invariant copy $e'$ in $G$. 
Since $e'$ is forcing in $G$, $e$ is forcing in $G_0$: two distinct perfect matchings of $G_0$ containing $e$ would lift to two distinct perfect matchings of $G$ containing $e'$.

We now assume that $G$ is simple. If $G\cong K_4$ or $\overline{C_6}$, then $G\in\mathcal F$. Otherwise, if $G$ is cubic, \cref{th1,note:cubic-near-bipartite} imply that $G\cong G_1$ or $G_3$. For the remainder of the simple-graph argument, assume that $G$ is not cubic.

We first consider the degrees of $u_i$ and $v_i$ for $i=1,2$.

\begin{Cla}
   \label{RMD}
   For $i=1,2$, $d_G(u_i)\leq 4$ and $d_G(v_i)\leq 4$.
\end{Cla}
\begin{proof}
      Suppose that at least one vertex in $V(R)$ has degree five or more, say $u_1$.
      By \cref{lem:invariant} (ii), $u_1$ is incident with at least two $b$-invariant edges of $G$, say $u_1v_3$ and $u_1v_4$.
      Then $u_1v_3$ and $u_1v_4$ are removable edges of $G$, and thus $\{v_3,v_4\}\subseteq V$.
      By the hypothesis, all $b$-invariant edges of $G$ are forcing edges.
      So $u_1v_3$ and $u_1v_4$ are two forcing edges of $G$ incident with $u_1$.
      \cref{N3} implies that these two forcing edges belong to two different
      triangles of $G$ containing $e_1=u_1u_2$, i.e. $v_iu_1u_2v_i$ for $i=3,4$,
      and $N_H(U_{=3}^{e_1})=N_H(u_2)=\{v_3, v_4\}$.
      By \cref{VGUV}, $V=\{v_1,v_2,v_3,v_4\}$ (at this stage, it is possible that $\{v_3,v_4\}\cap \{v_1,v_2\}\neq \emptyset$).
      Since $d_H(u_1)\geq 4$, $|U|=|V|=4$ and $N_H(u_1)=V$.

      If $d_G(v_1)\geq 4$ or $d_G(v_2)\geq 4$, then either $v_1$ or $v_2$ is incident with a
      $b$-invariant edge of $G$ by \cref{lem:invariant} (ii), say $v_1$; the incident $b$-invariant edge is also a forcing edge of $G$.
      By \cref{cl-e12}, $v_1\in N_H(U_{=3}^{e_1})$, contradicting that $N_H(U_{=3}^{e_1})=N_H(u_2)=\{v_3,v_4\}$.
      Thus, $d_G(v_1)=d_G(v_2)=3$.
      \begin{figure}[H]
         \centering
         \includegraphics[scale=0.3]{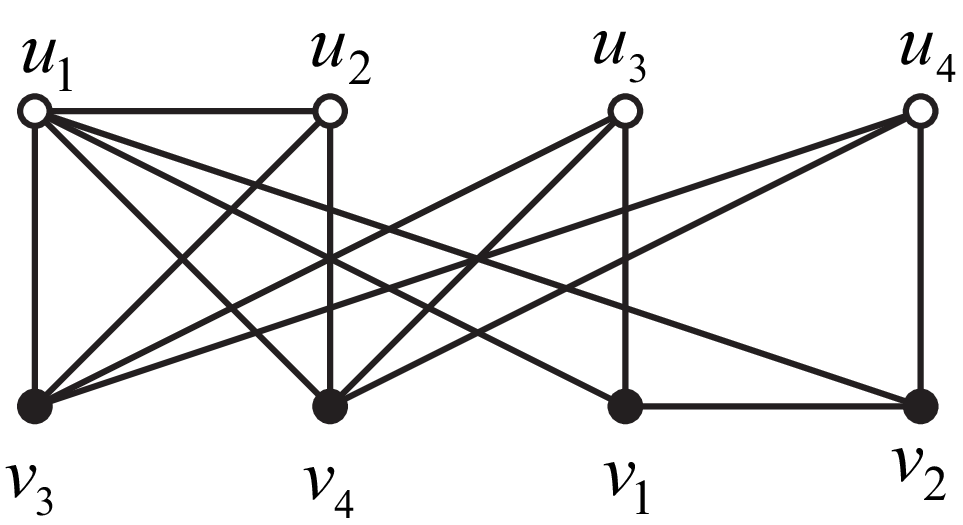}
         \caption{\label{f2-0}Illustration of \cref{RMD}: The graph $R_0$.}
      \end{figure}
      Since $$5+3\times 3 \leq \sum \limits_{x\in U} d_G(x)=\sum \limits_{y\in V} d_G(y)\leq 2\times 4+2\times 3=14,$$
      equality holds throughout, i.e. $d_G(u_i)=3$ and $d_G(v_i)=4$ for $i=3,4$.
      Let $U=\{u_1,u_2,u_3,u_4\}$. Since $N_G(v_i)=U$ ($i=3,4$),
      $\{v_3,v_4\}\subseteq N_G(u_i)$ for $i=3,4$.
      Recall that $u_1v_3$ is a forcing edge of $G$.
      Let $M_{u_1v_3}$ be the unique perfect matching of $G$ containing $u_1v_3$.
      Since $N_G(u_2)=\{u_1, v_3, v_4\}$, $u_2v_4\in M_{u_1v_3}$; hence $u_1v_4u_2v_3u_1$ is an $M_{u_1v_3}$-alternating 4-cycle.
      This implies that $G-\{u_1,v_4,u_2,v_3\}$ has a matching of size two, say $u_3v_1$ and $u_4v_2$.
      Since $d_G(u_i)=3$, $N_G(u_i)=\{v_3,v_4,v_{i-2}\}$, where $i=3,4$.
      It follows that $G\cong R_0$; see \cref{f2-0}.
      It can be checked that $u_i$ ($i=3,4$) is not incident with any forcing edge of $G$,
      and thus it is not incident with any $b$-invariant edge of $G$.
      By \cref{lem:invariant} (i), $u_i$ ($i=3,4$) belongs to a triangle of $G$, a contradiction.
   \end{proof}
It remains to prove the following claim to complete the classification of simple graphs.
\begin{Cla}
   \label{noncubic-classification}
   If $G$ is not cubic, then $G$ is isomorphic to one of $G_8$, $G_9$, and $G_{10}$.
\end{Cla}
\begin{proof}
   By \cref{RMD} and the fact that $\delta(G)\geq 3$, each of $u_1,u_2,v_1,v_2$ has degree three or four. Since $G$ is not cubic, it has a vertex of degree at least four. By \cref{lem:invariant}, this vertex is incident with a $b$-invariant edge $e$. Since neither edge of $R$ is removable, $e\in E(H)$. By hypothesis, $e$ is forcing, so \cref{cl-e12} implies that both $U_{=3}^{e_1}$ and $V_{=3}^{e_2}$ are nonempty. By interchanging $u_1$ with $u_2$, interchanging $v_1$ with $v_2$, and, if necessary, interchanging the roles of $U$ and $V$, it therefore suffices to consider the following three cases: (i) $d_G(u_1)=d_G(v_1)=3$ and $d_G(u_2)=d_G(v_2)=4$; (ii) $d_G(u_1)=d_G(v_1)=d_G(v_2)=3$ and $d_G(u_2)=4$; and (iii) $d_G(u_i)=d_G(v_i)=3$ for $i=1,2$.
\begin{Cas}
   \label{T34}
   $d_G(u_1)=d_G(v_1)=3$ and $d_G(u_2)=d_G(v_2)=4$.
\end{Cas}
   \begin{figure}[H]
      \centering
      \includegraphics[scale=0.3]{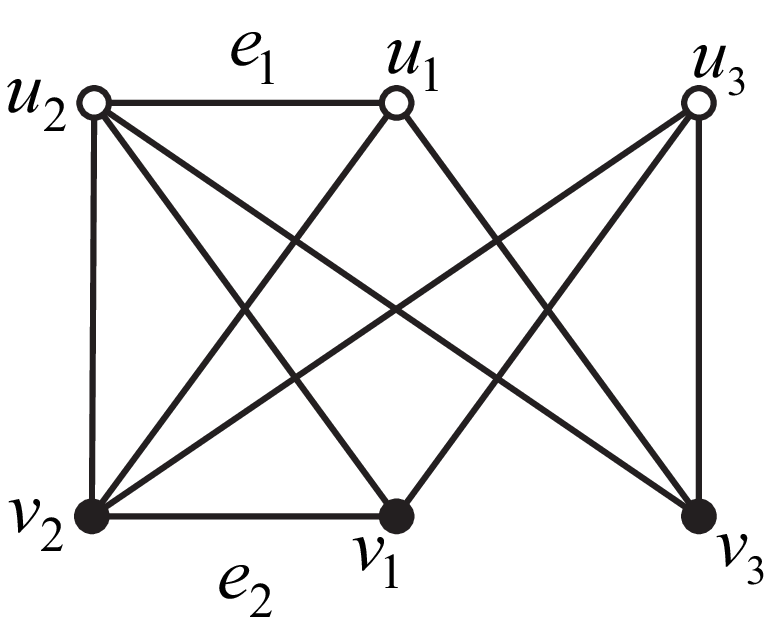}
      \caption{\label{IF1} Illustration of \cref{T34}: the graph $G_8$.}
   \end{figure}
   By \cref{lem:invariant} (ii), each of $u_2$ and $v_2$ is incident with a $b$-invariant edge of $G$, which is a forcing edge of $G$.
   By \cref{cl-e12}, $u_2\in N_H(V_{=3}^{e_2})=N_H(v_1)$ and $v_2\in N_H(U_{=3}^{e_1})=N_H(u_1)$.
   Since $d_H(u_2)=3$ and $N_H(u_2)\subseteq V$, $|V|\geq 3$.
   Since $d_H(u_1)=2$ and $v_2\in N_H(u_1)$, $|V(e_2)\cup N_H(u_1)|\leq 3$.
   By \cref{VGUV}, $|V|=|V(e_2)\cup N_H(U_{=3}^{e_1})|=|V(e_2)\cup N_H(u_1)|\leq 3$.
   So $|V|=3$. Since $|U|=|V|$, $|U|=3$.
   Let $U=\{u_1,u_2,u_3\}$ and $V=\{v_1,v_2,v_3\}$.
   Since $d_G(u_3)\geq 3$ and $d_G(v_3)\geq 3$, $N_G(u_3)=V$ and $N_G(v_3)=U$.
   It follows that $G\cong G_8$ (see \cref{IF1}).
\begin{Cas}
   \label{R3334}
   $d_G(u_1)=d_G(v_1)=d_G(v_2)=3$ and $d_G(u_2)=4$.
\end{Cas}
   Since $|U|=|V|$, the degree sums over $U$ and $V$ are equal, and $d_G(u_2)=4$ while all other degrees are at least three, $V$ contains a vertex of degree at least four.
   Since $d_G(v_1)=d_G(v_2)=3$, $V\setminus V(e_2)$ contains a vertex of degree at least four, and thus $|U|\geq 4$.
   Since $d_G(u_1)=3$, $|N_H(u_1)|\leq 2$.
   By \cref{VGUV}, $|V|=|N_H(U_{=3}^{e_1})\cup V(e_2)|=|N_H(u_1)\cup V(e_2)|\leq 4$.
   Since $|U|=|V|$, $|U|=|V|=4$.
   Let $U=\{u_1,u_2,u_4,u_3\}$ and $V=\{v_1,v_2,v_3,v_4\}$.
   Since $V=N_H(u_1)\cup \{v_1,v_2\}$,
   $$N_G(u_1)=\{u_2,v_3,v_4\}.$$

   By \cref{lem:invariant} (ii), $u_2$ is incident with at least one $b$-invariant edge of $G$, say $e$.
   Then $e$ is a forcing edge of $G$.
   By \cref{cl-e12}, $u_2\in N_H(V_{=3}^{e_2})=N_H(\{v_1,v_2\})$ and $V(e)\setminus \{u_2\}\subseteq N_H(U_{=3}^{e_1})=N_H(u_1)=\{v_3,v_4\}$.
   Relabel the vertices, if necessary, so that $u_2\in N_H(v_2)$ and $V(e)\setminus \{u_2\}=\{v_4\}$; that is, $u_2v_2\in E(G)$ and $e=u_2v_4$.
   Let $M_e$ be the unique perfect matching of $G$ containing $e$ (i.e. $u_2v_4$).
   Recall that $N_G(u_1)=\{u_2,v_3,v_4\}$.
   Then $u_1v_3\in M_e$. By the symmetry of $u_4$ and $u_3$, relabel the vertices so that
   $M_e=\{u_1v_3,u_2v_4,u_4v_1,u_3v_2\}$.
   Since $d_G(v_2)=3$, $N_G(v_2)=\{v_1,u_2,u_3\}$.
   Since $d_G(u_4)\geq 3$ and $v_2\notin N_G(u_4)$,
   $$N_G(u_4)=\{v_1,v_3,v_4\},$$ see \cref{IF8} (a).

   \begin{figure}[H]
      \centering
      \includegraphics[scale=0.26]{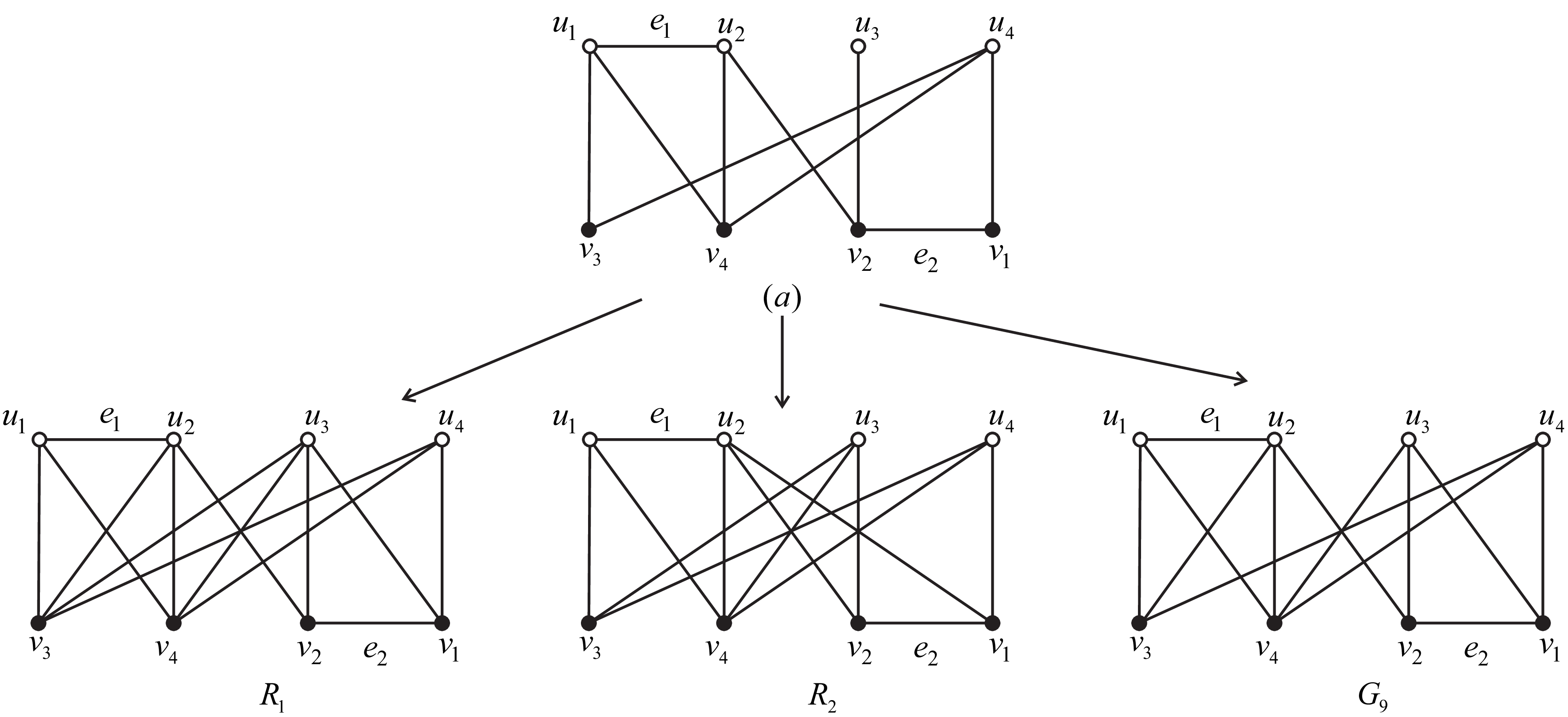}
      \caption{\label{IF8} Illustration of \cref{R3334}.}
   \end{figure}

   For the vertex $u_2$, $|N_H(u_2)|=3$.
   Since $\{v_2,v_4\}\subseteq N_H(u_2)$, either $u_2v_3\in E(G)$ or $u_2v_1\in E(G)$.
   If $u_2v_3\in E(G)$, we have
   $$N_G(u_2)=\{u_1,v_3,v_4,v_2\}.$$
   Since $d_H(v_1)=2$ and $\{u_1,u_2\}\cap N_H(v_1)=\emptyset$, $N_G(v_1)=\{v_2,u_4,u_3\}$.
   Next we consider the neighbors of $u_3$.
   If $d_G(u_3)=3$, then, by symmetry between $v_3$ and $v_4$, relabel the vertices so that $u_3v_4\in E(G)$.
   In this case, $$N_G(u_3)=\{v_4,v_1,v_2\}$$ and $G\cong G_9$.
   If $d_G(u_3)=4$, then $N_G(u_3)=V$ and thus $G\cong R_1$.
   If $u_2v_1\in E(G)$, we have $$N_G(u_2)=\{u_1,v_4,v_1,v_2\}$$ and $N_G(v_1)=\{v_2,u_4,u_2\}$.
   Since $d_G(u_3)\geq 3$ and $v_1\notin N_G(u_3)$, $$N_G(u_3)=\{v_2,v_3,v_4\}$$ and thus $G\cong R_2$.

   In the graph $R_i$ ($i=1,2$), $u_4$ is incident with at least one $b$-invariant edge of $G$ by \cref{lem:invariant} (i).
   However, $u_4$ is not incident with any forcing edge of $G$, a contradiction.
   It follows that $G\cong G_9$.
\begin{Cas}
   \label{R3333}
   $d_G(u_i)=d_G(v_i)=3$ for $i=1,2$.
\end{Cas}
   Since $G$ is not cubic, $V(G)\setminus V(R)$ contains a vertex of degree at least four.
   This implies that $|U|=|V|\geq 4$.
   By \cref{VGUV}, $|U|=|V|\leq 4+2=6$.
   \begin{Subcla}
      \label{clauv}
      $|U|=|V|\geq 5$.
   \end{Subcla}
   \begin{figure}[H]
      \centering
      \includegraphics[scale=0.3]{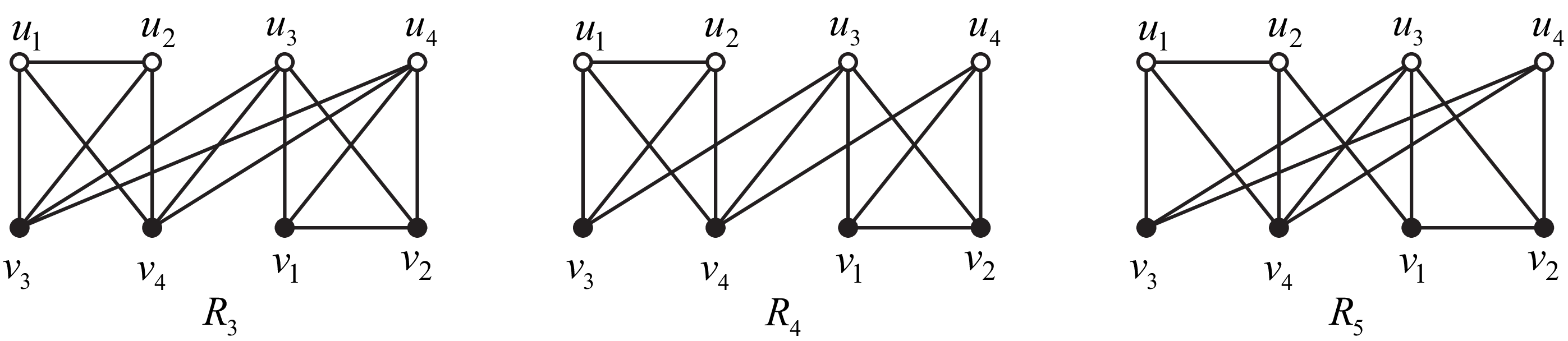}
      \caption{\label{f2-1} Illustration of \cref{clauv}.}
   \end{figure}
   \begin{proof}
   Suppose, to the contrary, that $|U|=|V|<5$. Then $|U|=|V|=4$. Let $U=\{u_1,u_2,u_3,u_4\}$ and $V=\{v_1,v_2,v_3,v_4\}$.
   Let $n_4$ be the number of 4-degree vertices in $U$.
   Since $d_G(u_1)=d_G(u_2)=3$, $n_4=1$ or $n_4=2$. If $n_4=2$, then $G\cong R_3$ (see \cref{f2-1}).
   We can check that $u_4v_4$ does not belong to any perfect matching of $H$,
   contradicting the fact that $H$ is matching covered.
   Next we consider the case that $n_4=1$.
   Then $U$ contains exactly one 4-degree vertex, say $u_3$,
   and $V$ contains exactly one 4-degree vertex, say $v_4$.
   If $E_G(V(e_1),V(e_2))\neq\emptyset$, say
   $u_2v_1\in E_G(V(e_1),V(e_2))$,
   then $G\cong R_5$; otherwise, $G\cong R_4$.
   For the graph $R_4$, $u_4v_4$ does not belong to
   any perfect matching of $H$, contradicting that $H$ is matching covered.
   For the graph $R_5$, $u_4$ is incident with at least one $b$-invariant edge of $G$ by \cref{lem:invariant} (i).
   However, $u_4$ is not incident with any forcing edge of $G$, a contradiction.
   Thus $|U|=|V|\geq 5$.
\end{proof}
\begin{Subcla}
   \label{clauv5}
   $|U|=|V|=5$.
\end{Subcla}
\begin{proof}
   Suppose, to the contrary, that $|U|=|V|\neq 5$. Then $|U|=|V|=6$ by \cref{clauv}.
   Let $U=\{u_i:1\leq i\leq 6\}$ and $V=\{v_i:1\leq i\leq 6\}$.
   By \cref{VGUV}, $U=V(e_1)\cup N_H(V_{=3}^{e_2})$ and $V=V(e_2)\cup N_H(U_{=3}^{e_1})$,
   and thus $N_H(x)\cap V(R)=\emptyset$ for each $x\in V(R)$.
   Relabel the vertices so that $N_H(u_1)=\{v_3,v_4\}$, $N_H(u_2)=\{v_5,v_6\}$, $N_H(v_1)=\{u_3,u_4\}$ and
   $N_H(v_2)=\{u_5,u_6\}$.

   By \cref{lem:invariant} (i), each of the vertices $v_3$, $v_4$, $u_5$, and $u_6$ is incident with at
   least one $b$-invariant edge of $G$, say $f_3$, $f_4$, $f_5$ and $f_6$, respectively.
   Then all of $f_3$, $f_4$, $f_5$ and $f_6$ are forcing edges of $G$.
   For $i=3,4,5,6$, let $M_{f_i}$ be the unique perfect matching of $G$ containing $f_i$.
   
   Suppose that $G$ contains a 4-cycle $C$ with $|V(C)\cap V(R)|=1$. 
   Relabel the vertices so that $C=u_2v_5u_3v_6u_2$.
   If $V(f_i)\subseteq \{v_3,v_4,u_5,u_6\}$ for some $3\leq i\leq 6$,
   we may assume by symmetry that $f_i=u_5v_3$.
   Since $M_{f_i}$ is a perfect matching of $H$, the prescribed neighborhoods force
   $u_1v_4$, $u_6v_2$, and $u_4v_1$ to belong to $M_{f_i}$.
   The remaining vertices are precisely those of $C$, so $C$ is an $M_{f_i}$-alternating
   4-cycle disjoint from $f_i$, contradicting that $f_i$ is forcing.
   Together with \cref{cl-e12},
   $$V(f_i)\setminus \{v_i\}\subseteq N_H(V_{=3}^{e_2})\setminus \{u_5,u_6\}=\{u_3,u_4\}, \mbox{for}~ i=3,4;$$ and
   $$V(f_j)\setminus \{u_j\}\subseteq N_H(U_{=3}^{e_1})\setminus \{v_3,v_4\}=\{v_5,v_6\}, \mbox{for}~  j=5,6.$$ 
   If $V(f_3)\cap V(f_4)\neq\emptyset$, we may assume that $f_3=u_3v_3$ and $f_4=u_3v_4$.
   Applying the preceding alternating-cycle argument to the 4-cycle $u_1v_3u_3v_4u_1$
   yields $V(f_j)\cap\{v_5,v_6\}=\emptyset$ for $j=5,6$, a contradiction.
   Thus $f_3$ and $f_4$ are disjoint. Similarly, $f_5$ and $f_6$ are disjoint.
   Relabel the vertices so that $f_i=u_iv_i$ for $3\leq i\leq 6$ (see \cref{f2-2}).
   Let $M_1=\{f_3,f_4,f_5,f_6,e_1,e_2\}$.
   Then $M_1$ is a perfect matching of $G$ containing $f_i$.
   Since $H$ is matching covered,
   $f_i$ belongs to a perfect matching of $H$, different from $M_1$.
   Thus $f_i$ is not a forcing edge of $G$, a contradiction.
   \begin{figure}[H]
      \centering
      \includegraphics[scale=0.3]{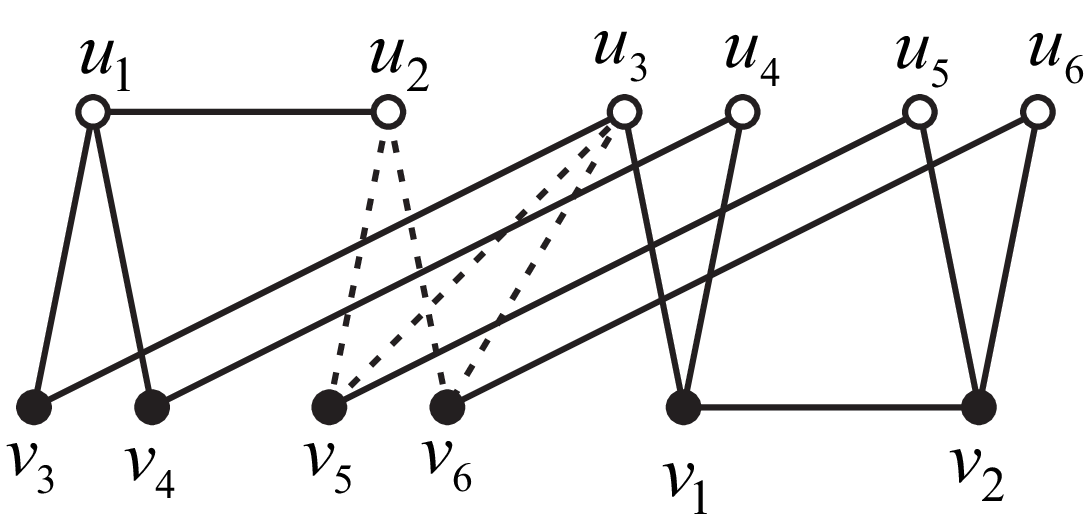}
      \caption{\label{f2-2} Illustration of \cref{clauv5}.}
   \end{figure}
  
   The preceding argument shows that $G$ contains no 4-cycle meeting $V(R)$ in exactly one vertex.
   Consequently, each vertex of $U\setminus V(R)$
   has at most one neighbor in each of $\{v_3,v_4\}$ and $\{v_5,v_6\}$.
   It also has exactly one neighbor in $\{v_1,v_2\}$, by the prescribed neighborhoods
   of $v_1$ and $v_2$. Thus its degree is at most three, and hence exactly three.
   The same argument applies to every vertex of $V\setminus V(R)$.
   Since the vertices in $V(R)$ also have degree three, $G$ is cubic, a contradiction.
   So $|U|=|V|=5$.
\end{proof}
\begin{Subcla}
   \label{TDT}
   $E_G(V(e_1),V(e_2))=\emptyset$.
   Moreover, $G$ contains two disjoint triangles.
\end{Subcla}
\begin{proof}
   By \cref{VGUV}, $U=V(e_1)\cup N_H(V_{=3}^{e_2})$.
   If $V(e_1)\cap N_H(V_{=3}^{e_2})=\emptyset$, this implies
   $E_G(V(e_1),V(e_2))=\emptyset$.
   Consequently, $V(e_2)\cap N_H(U_{=3}^{e_1})=\emptyset$.
   Since $|U|=|V|=5$ (by \cref{clauv5}), $|N_H(V_{=3}^{e_2})|=|N_H(U_{=3}^{e_1})|=3$.
   Since $U_{=3}^{e_1}=V(e_1)$ and $V_{=3}^{e_2}=V(e_2)$, $e_1$ and $e_2$ belong to two disjoint triangles, respectively.

   Next we assume that $V(e_1)\cap N_H(V_{=3}^{e_2})\neq \emptyset$.
   We show that this assumption leads to a contradiction.
   Since $|U|=|V|=5$, we have $|N_H(V_{=3}^{e_2})|=|N_H(U_{=3}^{e_1})|=4$.
   Let $U=\{u_i:1\leq i\leq 5\}$ and $V=\{v_i:1\leq i\leq 5\}$.
   Relabel the vertices so that $N_H(u_1)=\{v_3,v_4\}$, $N_H(u_2)=\{v_5,v_1\}$,
   $N_H(v_1)=\{u_2,u_3\}$ and $N_H(v_2)=\{u_4,u_5\}$, see \cref{f2-3}.

   If $u_3v_5$ is a $b$-invariant edge of $G$, then it is also a forcing edge of $G$.
   Let $M_{u_3v_5}$ denote the unique perfect matching of $G$ containing $u_3v_5$.
   Since $u_3v_5\in E(H)$ and $H$ is matching covered, $M_{u_3v_5}$ is also a perfect matching of $H$.
   Let $C=u_2v_1u_3v_5u_2$. Then $C$ is an $M_{u_3v_5}$-alternating 4-cycle.
   Since $H-V(C)$ is a bipartite graph with a unique perfect matching,
   it contains a 1-degree vertex in $U$ by \cref{B.U.PM}.
   However, $u_1$ retains both neighbors $v_3,v_4$ in $H-V(C)$, while each of $u_4,u_5$ has degree at least three in $H$ and loses at most one neighbor. Thus every vertex in $U\setminus V(C)$ has degree at least two in $H-V(C)$, a contradiction.

   \begin{figure}[H]
      \centering
      \includegraphics[scale=0.3]{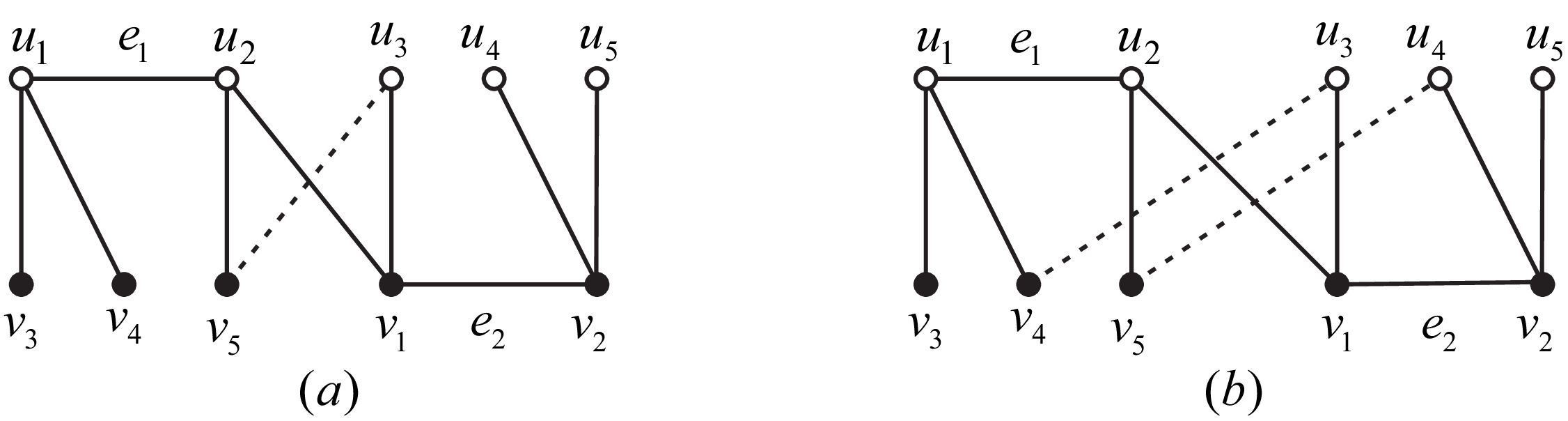}
      \caption{\label{f2-3} Illustration of \cref{TDT}.}
   \end{figure}
   By \cref{lem:invariant} (i), each of $u_3$ and $v_5$ is incident with at least one
   $b$-invariant edge of $G$, say $f_1$ and $f_2$, respectively.
   For $i=1,2$, $f_i$ is $b$-invariant and hence is a forcing edge in $E(H)$. Since $G$ contains no triangles, \cref{N3} gives $V(f_i)\cap V(R)=\emptyset$.
   From the above discussion, $f_1\neq f_2$ and we can assume that $f_1=u_3v_4$ and $f_2=u_4v_5$.
   Since $f_1$ is a forcing edge, $\{f_1,u_1v_3,u_2v_1,f_2,u_5v_2\}$ is the unique perfect matching of $G$ containing $f_1$; see \cref{f2-3} (b).
   By \cref{FMNR}, $f_2$ is a nonremovable edge of $G$, a contradiction.
   This completes the proof.
\end{proof}
By \cref{TDT}, relabel the vertices so that $N_H(u_1)=\{v_4,v_5\}$, $N_H(u_2)=\{v_4,v_3\}$,
$N_H(v_1)=\{u_4,u_5\}$ and $N_H(v_2)=\{u_4,u_3\}$. As shown in \cref{f2-4}, $u_4$ and $v_4$ belong to
two disjoint triangles of $G$, respectively.
\begin{Subcla}
   \label{TT33}
   $d_G(u_4)=d_G(v_4)=3$.
\end{Subcla}
\begin{proof}
   Suppose, to the contrary, that $d_G(u_4)\geq 4$ or $d_G(v_4)\geq 4$, say the former.
   By \cref{lem:invariant} (i), $u_4$ is incident with at most two non-$b$-invariant edges
   of $G$.
   By \cref{cl-e12}, each forcing edge of $G$ that lies in $E(H)$ has one end in $N_H(V_{=3}^{e_2})$ and one end in $N_H(U_{=3}^{e_1})$.
   Since $N_H(U_{=3}^{e_1})\cap \{v_1,v_2\}=\emptyset$ (by \cref{TDT}), $u_4v_1$ and $u_4v_2$
   are not forcing edges of $G$ and hence are not $b$-invariant.
   Let $f_1,f_2,\ldots,f_k$ be all edges of $G$ incident with $u_4$ other than
   $u_4v_1$ and $u_4v_2$.
   Then $k\geq 2$ and each $f_i$ ($1\leq i\leq k$) is a $b$-invariant edge of $G$, and thus a forcing edge of $G$.
   Let $M_R$ be a perfect matching of $G$ containing $R$.
   Then $M_R$ contains at least one edge from $\{f_1,f_2,\ldots,f_k\}$, say $f_{i_0}$.
   Since $f_{i_0}$ belongs to some perfect matching of $H$, which is different from
   $M_R$, $f_{i_0}$ is not a forcing edge of $G$, a contradiction.
\end{proof}
\begin{figure}[H]
   \centering
   \includegraphics[scale=0.3]{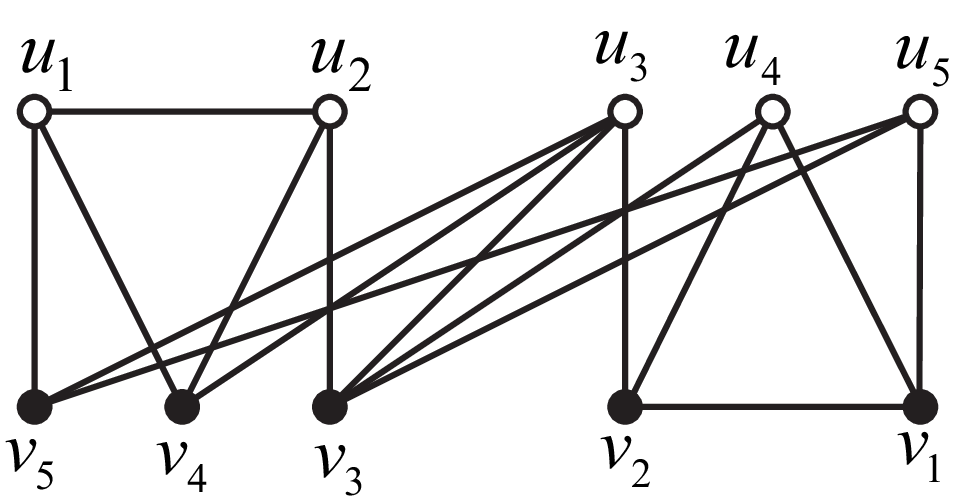}
   \caption{\label{f2-4} Illustration of \cref{R3333}: The graph $G_{10}$.}
\end{figure}
From the above discussion, all vertices of $G$, except for vertices in $\{u_3,u_5,v_3,v_5\}$, have degree three.
Since $G$ is not cubic, relabel the vertices so that $d_G(u_3)\geq 4$ and $d_G(v_3)\geq 4$.
The degree conditions give $N_G(v_4)=\{u_1,u_2,u_3\}$ and $N_G(u_4)=\{v_1,v_2,v_3\}$. Together with $\delta(G)\geq3$, these imply $N_G(u_5)=\{v_1,v_3,v_5\}$ and $N_G(v_5)=\{u_1,u_3,u_5\}$, so $d_G(u_5)=d_G(v_5)=3$. 
Thus $G\cong G_{10}$, completing \cref{R3333}. The three cases prove
\cref{noncubic-classification} and hence establish the classification
when $G$ is simple. 
\end{proof}

It remains to determine the possible parallel extensions. Let $G_0$ be the underlying simple graph of $G$, and let $D\subseteq E(G_0)$ consist of the edges having multiplicity greater than one in $G$. Each parallel copy is $b$-invariant and hence forcing in $G$. Thus every edge of $D$ is forcing in $G_0$, and no perfect matching of $G_0$ contains two edges of $D$. Moreover, some removable doubleton of $G_0$ is disjoint from $D$.

Inspection of the perfect matchings and removable doubletons of the seven underlying simple graphs shows that these conditions imply that, up to an automorphism, $D$ is contained in the dashed edge classes of one of the families in \cref{nf1}. Hence $G\in\mathcal F$, completing the proof of \cref{th2}.

\section*{Acknowledgements}
   This work was supported by the National Natural Science Foundation of China (Nos. 12271235 and 12601685), 
   the Startup Research Fund for Introduced Talents of Sichuan Normal University (No. kyqd20260308), 
   the Natural Science Foundation of Fujian Province (No. 2026J002034) 
   and Sichuan Science and Technology Program (No. 2026NSFSC2032).
\section*{Conflict of interest}
   The authors declare that they have no conflict of interest.
   
\section*{Data availability}
   No datasets were generated or analysed during the current study. A preprint of this article is available on arXiv at \url{https://doi.org/10.48550/arXiv.2607.00608}.

\end{document}